\documentclass[a4paper,12pt]{amsart}
\usepackage{amsmath,amsthm,amssymb}
\usepackage{mathrsfs}
\usepackage{mathtools}
\usepackage{ifthen}
\usepackage{booktabs}
\usepackage{caption}
\nonstopmode \numberwithin{equation}{section}
\newtheorem{thm}{Theorem}

\newtheorem{cor}{Corollary}
\newtheorem{lem}{Lemma}

\newtheorem{conj}{Conjecture}

\theoremstyle{definition}
\newtheorem{defn}{Definition}[section]

\newtheorem{prob}[equation]{Problem}

\newenvironment{rem}{%
\bigskip
\noindent \textsl{{\sl Remark. }}}{\bigskip}
\newenvironment{rems}{%
\bigskip
\noindent \textsl{{\sl Remarks. }}}{\bigskip}

\newcounter {own}
\def\theown {\thesection       .\arabic{own}}

\newenvironment{pf}[1][]{%
 \vskip 3mm
 \noindent
 \ifthenelse{\equal{#1}{}}%
  {{\slshape Proof. }}%
  {{\slshape #1.} }%
 }%
{\qed\bigskip}

\newcounter{alphabet}

\newcommand{\A}{{\mathcal A}}

\newcommand{\IC}{{\mathbb C}}

\newcommand{\D}{{\mathbb D}}

\def\be{\begin{equation}}
\def\ee{\end{equation}}

\newcommand{\bee}{\begin{enumerate}}
\newcommand{\eee}{\end{enumerate}}

\newcommand{\blem}{\begin{lem}}
\newcommand{\elem}{\end{lem}}
\newcommand{\bthm}{\begin{thm}}
\newcommand{\ethm}{\end{thm}}
\newcommand{\bcor}{\begin{cor}}
\newcommand{\ecor}{\end{cor}}
\newcommand{\beg}{\begin{examp}}
\newcommand{\eeg}{\end{examp}}
\newcommand{\begs}{\begin{examples}}
\newcommand{\eegs}{\end{examples}}
\newcommand{\bdefe}{\begin{defn}}
\newcommand{\edefe}{\end{defn}}
\newcommand{\bprob}{\begin{prob}}
\newcommand{\eprob}{\end{prob}}
\newcommand{\bei}{\begin{itemize}}
\newcommand{\eei}{\end{itemize}}

\newcommand{\bcon}{\begin{conj}}
\newcommand{\econ}{\end{conj}}
\newcommand{\bcons}{\begin{conjs}}
\newcommand{\econs}{\end{conjs}}
\newcommand{\bprop}{\begin{propo}}
\newcommand{\eprop}{\end{propo}}
\newcommand{\br}{\begin{rem}}
\newcommand{\er}{\end{rem}}
\newcommand{\brs}{\begin{rems}}
\newcommand{\ers}{\end{rems}}
\newcommand{\bo}{\begin{obser}}
\newcommand{\eo}{\end{obser}}
\newcommand{\bos}{\begin{obsers}}
\newcommand{\eos}{\end{obsers}}
\newcommand{\bpf}{\begin{pf}}
\newcommand{\epf}{\end{pf}}
\newcommand{\ba}{\begin{array}}
\newcommand{\ea}{\end{array}}
\newcommand{\beq}{\begin{eqnarray}}
\newcommand{\beqq}{\begin{eqnarray*}}
\newcommand{\eeq}{\end{eqnarray}}
\newcommand{\eeqq}{\end{eqnarray*}}

\newcounter{minutes}
\divide\time by 60
\newcounter{hours}
\multiply\time by 60 \addtocounter{minutes}{-\time}
\begin{document}
\title{Radius of convexity of certain classes of functions defined by convolution}
\begin{center}
{\tiny \texttt{FILE:~\jobname .tex,
        printed: \number\year-\number\month-\number\day,
        \thehours.\ifnum\theminutes<10{0}\fi\theminutes}
}
\end{center}
\author{Bappaditya Bhowmik${}^{~\mathbf{*}}$}
\address{Bappaditya Bhowmik, Department of Mathematics,
Indian Institute of Technology Kharagpur, Kharagpur - 721302, India.}
\email{bappaditya@maths.iitkgp.ac.in}
\author{Souvik Biswas}
\address{Souvik Biswas, Department of Mathematics, Indian Institute of Technology Kharagpur, Kharagpur - 721302, India.}
\email{souvikbiswas158@gmail.com}


\begin{abstract}
   Let $\mathcal{S}$ be the class of analytic univalent functions defined in the open unit disc $\D$ of the complex plane with the normalizations $f(0)=0$ and $f'(0)=1$. For $A\in (1,2]$, let $Co(A)$ denote the class of concave univalent functions defined in $\D$ with the opening angle $\pi A$ at infinity. In this article, by applying certain convolution techniques, we investigate the radius of convexity for the class $Co(A)\ast\mathcal{S}t(1/2)$, where $\mathcal{S}t(1/2)\subsetneq \mathcal{S}$ denotes the class of starlike functions of order $1/2$. Furthermore, we establish that the radius of convexity of the class $\mathcal{S}\ast\mathcal{S}t(1/2)$ is at least $0.19191$ (approximately). Here, `$\ast$' denotes the convolution (or Hadamard product) of two classes of functions.
\end{abstract}

\maketitle
\pagestyle{myheadings}
\markboth{B. Bhowmik and S. Biswas}{}
\maketitle
\pagestyle{myheadings}
\markboth{B. Bhowmik and S. Biswas }{On the radius of convexity of $Co(A)\ast\mathcal{S}t(1/2)$ and $\mathcal{S}\ast\mathcal{S}t(1/2)$}

\bigskip
\bigskip

\section{Introduction and Preliminaries}\label{P3sec1}
Let $\mathbb{C}$ be the complex plane and $\D:=\{z\in\mathbb{C} : |z|<1\}$ be the open unit disc in $\mathbb{C}$. We denote the unit circle by $\partial \D=\{z\in\mathbb{C} : |z|=1\}$. We consider the class $\mathcal{A}$ consisting of all analytic functions in $\D$ that satisfy the normalization $f(0)=0=f'(0)-1$. Let $\mathcal{S}$ denote the class of univalent functions in $\mathcal{A}$. Over the years, researchers have investigated various subclasses of $\mathcal{S}$ with specific geometric properties. Among these, the most notable are the class of convex functions of order $\beta$, $0\leq \beta<1$, defined as
$$
  \mathcal{C}(\beta)=\{f\in \mathcal{A} : {\rm{Re}}\,\left(1+\frac{zf''(z)}{f'(z)}\right)>\beta, \quad z\in\D\},
$$
and the class of starlike functions of order $\beta$, $0\leq \beta<1$, defined as 
$$
  \mathcal{S}t(\beta)=\{f\in \mathcal{A} : {\rm{Re}}\,\left(\frac{zf'(z)}{f(z)}\right)>\beta, \quad z\in \D\}.
$$
We denote the class of all convex functions by $\mathcal{C}:=\mathcal{C}(0)$, which consists of all functions $f\in\mathcal{S}$ such that $f$ maps $\D$ conformally onto a convex domain. Similarly, the class of all starlike functions is denoted by $\mathcal{S}t:=\mathcal{S}t(0)$, which consists of all functions $f \in \mathcal{S}$ such that $f$ maps $\D$ conformally onto a starlike domain with respect to the origin. Since every convex function is starlike, a function which is convex in $\D$ must satisfy ${\rm{Re}}\,\left(zf'(z)/f(z)\right)>0$, $z\in\D$. In fact, for $f\in \mathcal{C}$, this condition can be improved to ${\rm{Re}}\,\left(zf'(z)/f(z)\right)>1/2$ for $z\in\D$ (see \cite{str}). In other words, $\mathcal{C}\subsetneq \mathcal{S}t(1/2)$. We recall that $f\in\A$ is said to be a close-to-convex function if there exists $g \in \mathcal{C}$ such that 
$$
 {\rm{Re}}\,\left(\frac{f'(z)}{g'(z)}\right)>0, \quad z\in\D.
$$
It is well-known that $\mathcal{C}\subsetneq \mathcal{S}t(1/2)\subsetneq \mathcal{S}t\subsetneq \mathcal{K}\subsetneq \mathcal{S}$. In this article, we address the classical radius problem in geometric function theory. The {\it{radius of convexity}} ({\it{starlikeness}}, respectively) of a subclass $\A_1$ of $\A$ is the largest number $r\in (0,1]$ such
that every function $f\in \A_1$ is convex (starlike, respectively) in $\D_r=\{z\in \IC: |z|<r\}$. In $1920$, Nevanlinna (\cite{nev}) proved that the radius of convexity of $\mathcal{S}$ is $2-\sqrt{3}$. Later in $1934$, Grunsky (see \cite[p.~141]{grunsky}) showed that the radius of starlikeness of $\mathcal{S}$ is $\tanh{\pi/4}$. Let $Co(A)$ be the class of functions $f$ that satisfy the following conditions:
 \begin{itemize}
     \item[(i)] $f\in\mathcal{S}$ with the additional condition $f(1)=\infty$.
     \item[(ii)] $\mathbb{C}\setminus f(\D)$ is convex.
     \item[(iii)] The opening angle of $f(\D)$ at $\infty$ is less than or equal to $\pi A$, $A\in (1,2]$.
 \end{itemize}
This class is called the class of concave univalent functions with opening angle $\pi A$, $A\in(1,2]$, at infinity. In \cite[Theorem~2]{fg}, Avkhadiev and Wirths proved that $f \in Co(A)$ if and only if
$$
 {\rm{Re}}\,T_f(z)>0, \quad z \in \D,
$$
where, $f(0)=0=f'(0)-1$ and
\begin{equation}\label{P3eq1.3}
T_f(z):=\frac{2}{A-1}\left[\frac{A+1}{2}\left(\frac{1+z}{1-z}\right)-1-z\frac{f''(z)}{f'(z)}\right].
\end{equation}
For more details about this class, we urge readers to go through the articles \cite{av, fg, bhowmik, charac}. In order to describe the main goal of this article, we need to be familiar with the concept of convolution (or Hadamard product). The convolution of two functions $f,g\in \mathcal{A}$ with the power series expansions $f(z)=\sum_{n=0}^\infty a_nz^n$ and $g(z)=\sum_{n=0}^\infty b_nz^n$ is defined as 
$$
 (f\ast g)(z)=\sum_{n=0}^\infty a_n b_n z^n\in \A.
$$
In \cite{sheil}, Ruscheweyh and Sheil-Small proved the P\'olya-Schoenberg conjecture which states that the class $\mathcal{C}$ is closed under convolution, i.e. $\mathcal{C}\ast \mathcal{C}\subseteq \mathcal{C}$. Furthermore, in \cite{sheil}, it was obtained that the class $\mathcal{S}t (1/2)$ is closed under convolution, and $\mathcal{C}\ast \mathcal{K} \subseteq \mathcal{K}$.
In \cite{small}, it is shown that for $f,g\in \mathcal{S}t$, $f\ast g$ need not be in $\mathcal{S}$. Later, in $1997$, Y. Ling and S. Ding (see \cite{ling}) obtained the radius of starlikeness and convexity of the class $\mathcal{S}t \ast \mathcal{S}t$ as $2-\sqrt{3}$ and $5-2\sqrt{6}$, respectively. In $2003$, Richard Greiner and Oliver Roth proved that the radius of convexity of the class $\mathcal{S} \ast \mathcal{K}$ is $5-2\sqrt{6}$ (see \cite[Theorem~2.1]{roth}) and conjectured that the radius of convexity of the class $\mathcal{S} \ast \mathcal{S}$ is $5-2\sqrt{6}$ (see \cite[Conjecture~$2.3$]{roth}), which is still open. In this article, we explore convolution of functions in $Co(A)$ with $\mathcal{C}$ and $\mathcal{S}t (1/2)$. Since $\mathcal{C} \ast \mathcal{C} \subseteq \mathcal{C}$ and the radius of convexity of $Co(A)$ is $A-\sqrt{A^2-1}$ (see \cite[Corollary~2.13]{bhowmik}), it is readily seen that every function in $Co(A)\ast\mathcal{C}$ is convex in $|z|<A-\sqrt{A^2-1}$. This is best possible which can be seen as follows. If we consider
$$
  f(z)=\frac{1}{2A}\left[\left(\frac{1+z}{1-z}\right)^A-1\right]\in Co(A), ~ A\in(1,2] \quad {\rm{and}} \quad g(z)=\frac{z}{1-z}\in\mathcal{C},
$$
then $(f\ast g)(z)=f(z)$, which is convex in $|z|<A-\sqrt{A^2-1}$ but fails to be convex in any larger disk (see \cite[Corollary~2.13]{bhowmik}). Next, we see that a lower bound for the radius of convexity of the class $Co(A)\ast \mathcal{S}t(1/2)$ is $\left(A-\sqrt{A^2-1}\right)\sqrt{2\sqrt{3}-3}$. This can be seen by noting that $\mathcal{C}\ast \mathcal{C} \subseteq \mathcal{C}$, together with the fact that the radii of convexity of the classes $Co(A)$ and $\mathcal{S}t(1/2)$ are $A-\sqrt{A^2-1}$ and $\sqrt{2\sqrt{3}-3}$ (see \cite[Theorem~1]{mac}), respectively. In this article (Theorem~\ref{P3thm1}), using some specific convolution techniques, we obtain an improved lower bound for the radius of convexity of the class $Co(A)\ast \mathcal{S}t(1/2)$ for $A\in[1.370794,2]$. Furthermore, we consider a more general class $\mathcal{S}\ast \mathcal{S}t(1/2)$ and investigate its radius of convexity. It is important to note that the radius of convexity of the class $\mathcal{S} \ast \mathcal{S}t(1/2)$ is at most $2-\sqrt{3}\approx 0.267949$. This can be seen by noting that $(k\ast g)(z)=k(z)$ is convex in $|z|<2-\sqrt{3}$ but fails to be convex in any larger disk, where
$$
 k(z)=\frac{z}{(1-z)^2}\in\mathcal{S} \quad {\rm{and}} \quad g(z)=\frac{z}{1-z}\in\mathcal{S}t(1/2).
$$
In $2003$, Richard Greiner and Oliver Roth proved that the radius of convexity of the class $\{(1-\lambda)f(z)+\lambda zf'(z):f\in\mathcal{S}\}$ for fixed $\lambda\in\mathbb{C}$ with $\left|2\lambda-1\right|\leq 1$ is $5-2\sqrt{6}$ (see \cite[Theorem~1.2]{roth}). As an application of this result, the authors obtained that the radius of convexity of $\mathcal{S}\ast\mathcal{K}$ is $5-2\sqrt{6}$ (see \cite[Theorem~2.1]{roth}). Since $\mathcal{S}t(1/2)\subsetneq \mathcal{K}$, it then follows that every function in $\mathcal{S}\ast\mathcal{S}t(1/2)$ maps $|z|<5-2\sqrt{6}\approx 0.10102$ onto a convex domain. By a direct computation, we see that a lower bound for the radius of convexity of the class $\mathcal{S} \ast \mathcal{S}t(1/2)$ is $(2-\sqrt{3})\sqrt{2\sqrt{3}-3} \approx 0.18254$. This can be seen by noting that $\mathcal{C}\ast \mathcal{C} \subseteq \mathcal{C}$ together with the fact that the radii of convexity of the classes $\mathcal{S}$ and $\mathcal{S}t(1/2)$ are $2-\sqrt{3}$ and $\sqrt{2\sqrt{3}-3}$, respectively. In this article (Theorem~\ref{P3thm2}), we obtain an improved lower bound for the radius of convexity of the class $\mathcal{S} \ast \mathcal{S}t(1/2)$, which is approximately $0.19191$.
We must emphasize here that, the technique used to prove Theorems~\ref{P3thm1} and \ref{P3thm2} of this article is completely different than the one which was used by Richard Greiner and Oliver Roth to obtain the radius of convexity for the class $\mathcal{S} \ast \mathcal{K}$ in \cite[Theorem~2.1]{roth}.

\section{Radius of convexity of  $Co(A)\ast \mathcal{S}t(1/2)$}

We need Lemmas~\ref{P3lem1}, \ref{P3lem2} and \ref{P3lem3} in order to establish the main result of this section.
\begin{lem}\label{P3lem1}
     Let $\phi$ and $g$ be analytic functions defined in $\D$ and satisfy $\phi(0)=0=g(0)$ with $\phi'(0)\ne 0$, $g'(0)\ne 0$. If for each complex number $\sigma \in \partial \D$ and $\alpha \in \partial \D$,
     \begin{equation}\label{P3eq2.1}
          \phi(z) \ast \left(\frac{R_1+\alpha \sigma Rz}{R_1-\sigma Rz} g(Rz)\right) \ne 0, \quad 0<|z|<1,
     \end{equation}
    where $0<R\leq R_1<1$, then 
    $$
     {\rm{Re}}\, \left(\frac{\phi(z) \ast \left(\frac{R_1+\sigma Rz}{R_1-\sigma Rz}\,g(Rz)\right)}{\phi(z)\ast g(Rz)}\right)>0, \quad z\in\D.
    $$
\end{lem}
\begin{pf}
     We first observe that if $\alpha=-1$, then from \eqref{P3eq2.1} we get
    $$
     \phi(z) \ast g(Rz) \ne 0, \quad 0<|z|<1.
    $$
    For each complex number $\sigma \in \partial \D$ and $\alpha \in \partial \D$, we have
    \begin{align*}
         &\phi(z) \ast \left(\frac{R_1+\alpha \sigma Rz}{R_1-\sigma Rz}\,g(Rz)\right)\\=& \left(\frac{1+\alpha}{2}\right)\phi(z) \ast \left(\frac{R_1+\sigma Rz}{R_1-\sigma Rz}\,g(Rz)\right)
         +\left(\frac{1-\alpha}{2}\right) \phi(z) \ast g(Rz), \quad z\in \D.
    \end{align*}
   Dividing the above equation by $\phi(z) \ast g(Rz)$, we get
   $$
     \left(\frac{1+\alpha}{2}\right)\frac{\phi(z) \ast \left(\frac{R_1+\sigma Rz}{R_1-\sigma Rz}\,g(Rz)\right)}{\phi(z)\ast g(Rz)}=\frac{ \phi(z) \ast \left(\frac{R_1+\alpha \sigma Rz}{R_1-\sigma Rz}\,g(Rz)\right)}{\phi(z) \ast g(Rz)}-\frac{1-\alpha}{2},\quad 0<|z|<1.
   $$
   If we assume $\alpha \ne -1$, then by \eqref{P3eq2.1}, and, from the above equation we get
    $$
     \frac{\phi(z) \ast \left(\frac{R_1+\sigma Rz}{R_1-\sigma Rz}\,g(Rz)\right)}{\phi(z)\ast g(Rz)} \ne -\frac{1-\alpha}{1+\alpha}, \quad 0<|z|<1.
    $$
    Thus, the function in the left-hand side of the above inequation does not take any value on the imaginary axis, but, clearly has the value $1$ at $z=0$. This proves the lemma.
\end{pf}

\begin{lem}\label{P3lem2}
    Let $h$ be an analytic function in $\mathbb{D}$ with $h(0)=0$, and there exists $\beta\in \partial \D$ such that 
    $$
     {\rm{Re}}\,\left((R_1-\beta z)\frac{h(z)}{z}\right)>0, \quad |z|<R,
    $$
    where $0<R\leq R_1<1$. Then for each $\phi \in \mathcal{S}t(1/2)$,
    $$
     \phi(z) \ast h(Rz) \ne 0, \quad 0<|z|<1.
    $$
\end{lem}
\begin{pf}
    Let $0<R\leq R_1<1$. Then from the given hypothesis, we see that
    $$
     {\rm{Re}}\,\left((R_1-\beta Rz)\frac{h(Rz)}{Rz}\right)>0, \quad z\in\D.
    $$
    By the Herglotz formula (see \cite{her}), there exists $\gamma\in\partial \D$, $\gamma \ne -1$ such that 
    $$
     \frac{h(Rz)}{R_1 h'(0)}=\int\limits_{\partial \D} \frac{Rz(1+\gamma \sigma z)}{(1-\sigma z)(R_1-\beta Rz)}~ d\mu (\sigma), \quad z\in\D,
    $$
    where $\mu$ is a probability measure on $\partial \D$. Thus,
    $$
     \frac{1}{R_1h'(0)} \phi(z) \ast h(Rz)=\int\limits_{\partial \D} \phi(z)\ast\left(\frac{Rz(1+\gamma \sigma z)}{(1-\sigma z)(R_1-\beta Rz)}\right)~ d\mu (\sigma),
    $$
    which can be written as
    \begin{align*}
        &\frac{1}{R_1h'(0)} \phi(z) \ast h(Rz)\\=&\int\limits_{\partial \D} \phi(z)\ast \left((\gamma+1)Rz(R_1-\beta Rz)^{-1}(1-\sigma z)^{-1}-\gamma Rz(R_1-\beta Rz)^{-1} \right)~ d\mu(\sigma).
    \end{align*}
    Then by a little computation, we see that
    \begin{align}
        &\frac{1}{R_1h'(0)}\phi(z) \ast h(Rz)\label{P3eq2.2}\\&=\frac{1}{\beta}\phi\left(\frac{\beta Rz}{R_1}\right)\left((\gamma+1)\int\limits_{\partial \D}\frac{\phi(z)\ast Rz(R_1-\beta Rz)^{-1}(1-\sigma z)^{-1}}{\phi(z)\ast Rz(R_1-\beta Rz)^{-1}}d\mu(\sigma)-\gamma \right)\nonumber.
    \end{align}
   
  Since $\phi \in \mathcal{S}t(1/2)$, from \cite[p.~123]{sheil} we have 
   $$
    {\rm{Re}}\,\left(\frac{\phi(z)\ast z(1-\sigma_1 z)^{-1}(1-\sigma_2 z)^{-1}}{\phi(z)\ast z(1-\sigma_1 z)^{-1}}\right)>\frac{1}{2}, \quad |z|<1,
   $$
   for any pair of complex numbers $\sigma_i$, $i=1,2$ satisfying $|\sigma_i|\leq 1$, $i=1,2$. If we choose $\sigma_1=\beta R/{R_1}$ and $\sigma_2=\sigma$, then we get
   $$
    {\rm{Re}}\,\left(\frac{\phi(z)\ast z(R_1-\beta Rz)^{-1}(1-\sigma z)^{-1}}{\phi(z)\ast z(R_1-\beta Rz)^{-1}}\right)>\frac{1}{2}, \quad |z|<1.
   $$
   This implies that the expression in the second parentheses of the right hand side of \eqref{P3eq2.2} cannot vanish in $0<|z|<1$. Since $\phi(\beta Rz/{R_1})\ne 0$ $(0<|z|<1)$, the lemma follows.
\end{pf}

\begin{lem}\label{P3lem3}
    If $\phi \in \mathcal{S}t (1/2)$ and $f\in Co(A)$, $A\in(1,2]$, then for each complex number $\sigma\in\partial \D$ and $\alpha\in\partial \D$, we have
    $$
     \phi(z)\ast \left(\frac{R_0+\alpha \sigma R_2 z}{R_0-\sigma R_2 z}\,R_2zf'(R_2z)\right) \ne 0, \quad 0<|z|<1,
    $$
    where $R_0:=A-\sqrt{A^2-1}$ and $R_2$ is the least value of $r\in(0,R_0)$ satisfying the equation $\xi(r)=0$ with
    $$
     \xi(r):=\sin^{-1}\left(\frac{r}{R_0}\right)+2A\sin^{-1}(r)-\frac{\pi}{2}.
    $$
\end{lem}
\begin{pf}
    According to Lemma~\ref{P3lem2}, it is sufficient to prove that for each complex number $\sigma\in\partial \D$ and $\alpha\in\partial \D$, there exists a constant $\beta\in\partial \D$ such that
  \begin{equation}\label{P3eq2.3}
      {\rm{Re}}\,\left((R_0-\beta z)\frac{R_0+\alpha \sigma z}{R_0-\sigma z}f'(z)\right)>0, \quad |z|<R_2,
  \end{equation}
  where $R_0$ and $R_2$ are defined in the statement of the lemma. A straightforward calculation yields
  \begin{equation}\label{P3eq2.4}
      \left|\arg{\left((R_0+\alpha \sigma z)f'(z)\right)}\right| \leq \sin^{-1}\left(\frac{r}{R_0}\right)+\left|\arg{f'(z)}\right|, \quad |z|=r<1.
  \end{equation}
  Since $f\in Co(A)$, we have $\left|\arg{f'(z)}\right|\leq 2A \sin^{-1}(|z|)$ for $z\in\D$ (see \cite[Corollary~2.5]{bhowmik}). Thus, from \eqref{P3eq2.4} we get
  \begin{equation*}
      \left|\arg{\left((R_0+\alpha \sigma z)f'(z)\right)}\right| \leq \sin^{-1}\left(\frac{r}{R_0}\right)+2A\sin^{-1}(r), \quad |z|=r<1.
  \end{equation*}
  The right hand side of the above inequality is less than $\pi/2$ if $|z|<R_2$.  Thus,
  \begin{equation}\label{P3eq2.5}
      {\rm{Re}}\,\left((R_0+\alpha \sigma z)f'(z)\right)>0, \quad |z|<R_2.
  \end{equation}
  We now investigate the existence of $R_2$ for each $A\in(1,2]$. We see that the function $\xi$ which is defined in the statement of the lemma, is continuous on $[0,R_0]$ with
  $$
     \xi(0)=-\frac{\pi}{2}<0 \quad {\rm{and}} \quad \xi(R_0)=2A\sin^{-1}(R_0)>0.
  $$
  Therefore, by the intermediate value theorem, $\xi$ has at least one root in $(0,R_0)$. Hence, $R_2$ exists for each $A\in(1,2]$. Thus, if we consider $\beta=\sigma$, then from \eqref{P3eq2.5} we get
  $$
    {\rm{Re}}\,\left((R_0-\beta z)\frac{R_0+\alpha \sigma z}{R_0-\sigma z}f'(z)\right)>0, \quad |z|<R_2.
  $$
  Hence, there exists a constant $\beta \in \partial \D$ such that the inequality \eqref{P3eq2.3} holds. This proves the lemma.
\end{pf}

We now establish the main result of this section. In the following theorem, we determine a lower bound for the radius of convexity of the class $Co(A)\ast \mathcal{S}t(1/2)$.
\begin{thm}\label{P3thm1}
    Let $A\in(1,2]$. Then the radius of convexity of the class $Co(A)\ast\mathcal{S}t(1/2)$ is at least $R_2$, where $R_2$ is defined as in \textnormal{Lemma~\ref{P3lem3}}.
\end{thm}
\begin{pf}
    It is easy to see that for $f\in Co(A)$ and $\phi\in\mathcal{S}t(1/2)$,
    \begin{equation}\label{P3eq2.6}
       \frac{Rz(f\ast \phi)''(Rz)}{(f \ast \phi)'(Rz)}=\frac{\phi(z)\ast \big((Rz)^2f''(Rz)\big)}{\phi(z)\ast \big(Rzf'(Rz)\big)}, \quad 0<|z|<1,
    \end{equation}
    where $0<R<1$. If we define 
    $$
      F(z):=1+\frac{zf''(z)}{f'(z)}, \quad z \in \D,
    $$
    and
    $$
      G(z):=zf'(z), \quad z \in \D,
    $$
    then from \eqref{P3eq2.6} we get
    \begin{equation}\label{P3eq2.7}
       1+\frac{Rz(f\ast \phi)''(Rz)}{(f \ast \phi)'(Rz)}=\frac{\phi(z)\ast \big((FG)(Rz)\big)}{\phi(z)\ast G(Rz)}, \quad 0<|z|<1.
    \end{equation}
    Since $f\in Co(A)$ and the radius of convexity of $Co(A)$ is $A-\sqrt{A^2-1}$ (see \cite[Corollary~2.13]{bhowmik}), we have
    \begin{equation}\label{P3eq2.8}
     {\rm{Re}}\,F(z)>0, \quad |z|<R_0:=A-\sqrt{A^2-1}.
    \end{equation}
    Let $w=R_0 z$, $z\in\D$. Since $|w|<R_0$, from \eqref{P3eq2.8} we get
    $$
     {\rm{Re}}\,F(w)>0,
    $$
    which implies
    $$
     {\rm{Re}}\,F(R_0z)>0, \quad z\in\D.
    $$
    Since $F(0)=1$, by the Herglotz formula we have
    $$
     F(R_0z)=\int\limits_{\partial \D} \frac{1+\sigma z}{1-\sigma z} d\mu (\sigma), \quad z\in\D,
    $$
    where $\mu$ is a probability measure on $\partial \D$. The above equation can be rewritten as
    $$
     F(z)=\int\limits_{\partial \D} \frac{R_0+\sigma z}{R_0-\sigma z} d\mu (\sigma), \quad |z|<R_0.
    $$
    By virtue of simple computations, from the above equation we get
    \begin{equation*}
        \frac{\phi(z)\ast \big((FG)(Rz)\big)}{\phi(z)\ast G(Rz)}=\frac{1}{\phi(z)\ast G(Rz)}\int\limits_{\partial \D}\phi(z) \ast \left(\frac{R_0+\sigma Rz}{R_0-\sigma Rz}\,G(Rz)\right)d\mu (\sigma), \quad z\in\D,
    \end{equation*}
    where $0<R\leq R_0$. Thus, from \eqref{P3eq2.7} and the above equation we get
    \begin{equation}\label{P3eq2.9}
        1+\frac{Rz(f\ast \phi)''(Rz)}{(f \ast \phi)'(Rz)}=\frac{1}{\phi(z)\ast G(Rz)}\int\limits_{\partial \D}\phi(z) \ast \left(\frac{R_0+\sigma Rz}{R_0-\sigma Rz}\,G(Rz)\right)d\mu (\sigma).
    \end{equation}
    By Lemmas~\ref{P3lem1} and \ref{P3lem3}, it follows that 
    $$
     {\rm{Re}}\, \left(\frac{\phi(z) \ast \left(\frac{R_0+\sigma R_2z}{R_0-\sigma R_2z}\,G(R_2z)\right)}{\phi(z)\ast G(R_2z)}\right)>0, \quad z\in\D,
    $$
    where $R_2$ is defined as in Lemma~\ref{P3lem3}. Using this, from \eqref{P3eq2.9} we get
    $$
     {\rm{Re}}\,\left(1+\frac{R_2z(f\ast \phi)''(R_2z)}{(f \ast \phi)'(R_2z)}\right)>0, \quad z\in\D,
    $$
    i.e.
    $$
     {\rm{Re}}\,\left(1+\frac{z(f\ast \phi)''(z)}{(f \ast \phi)'(z)}\right)>0, \quad |z|<R_2.
    $$
    Therefore, $f\ast \phi$ is convex in $|z|<R_2$ for every $f\in Co(A)$ and for every $\phi\in\mathcal{S}t(1/2)$. This completes the proof of the theorem.
\end{pf}

\begin{rem}
As we have discussed in Section~\ref{P3sec1}, a previously known lower bound for the radius of convexity of $Co(A)\ast\mathcal{S}t(1/2)$ is $R':=\left(A-\sqrt{A^2-1}\right)\sqrt{2\sqrt{3}-3}$. The lower bound $R_2$ obtained in Theorem~\ref{P3thm1} satisfies the equation $\xi(r)=0$, where $\xi$ is defined as in the statement of Lemma~\ref{P3lem3}. It is easy to see that for fixed $A\in(1,2]$, 
$$
 \xi'(r)>0, \quad r\in\left(0,A-\sqrt{A^2-1}\right).
$$
Since the function $\xi$ is strictly increasing on $\left(0,A-\sqrt{A^2-1}\right)$ and
$$
 \xi(R')< 0=\xi(R_2) \quad {\rm{for}} \quad A\in[1.370794,2],
$$
we have
$$
 R'< R_2.
$$
Thus, in Theorem~\ref{P3thm1} we obtained an improved lower bound $R_2$ for the radius of convexity of $Co(A)\ast\mathcal{S}t(1/2)$ for $A\in[1.370794,2]$. In Table~\ref{P3tab1}, we list the values of $R'$ and $R_2$ for different values of $A\in[1.370794,2]$.
\captionsetup[table]{skip=8pt, width=\textwidth, justification=centering}
    \begin{table}[htbp]
    \centering
    \renewcommand{\arraystretch}{1.2}
 	\begin{tabular}{{c@{\hspace{1cm}}c@{\hspace{1cm}}c}}
    \toprule
 		  \textbf{Values of $A$} & \textbf{Values of $R'$} & \textbf{Values of $R_2$}\\
        \midrule
              $1.370794$ & $0.295119812$ & $0.295119816$\\
              $1.5$ & $0.260214360$ & $0.264920163$\\
              $1.7$ & $0.221561105$ & $0.229515864$\\
              $1.9$ & $0.193781954$ & $0.202920564$\\
 			 $2$ &   $0.182540398$ & $0.191909667$\\
        \bottomrule
 	\end{tabular}
    \caption{Values of $R'$ and $R_2$ for different values of $A\in[1.370794,2]$.}
    \label{P3tab1}
 	\end{table}
\end{rem}

\section{Radius of convexity of $\mathcal{S}\ast\mathcal{S}t(1/2)$}

In this section, we explore the convolution $\phi \ast f$, where $\phi \in \mathcal{S}$ and $f\in \mathcal{S}t(1/2)$. Here, we focus on determining the radius of convexity of the class $\mathcal{S}\ast\mathcal{S}t(1/2)$. Since $Co(A)\subsetneq\mathcal{S}$, we see that this result is more general than that was established in Theorem~\ref{P3thm1}. To proceed, we first establish an essential lemma as before.
\begin{lem}\label{P3lem4}
    If $\phi \in \mathcal{S}t (1/2)$ and $f\in \mathcal{S}$, then for each complex number $\sigma\in\partial \D$ and $\alpha\in\partial \D$, we have
    $$
     \phi(z)\ast \left(\frac{R_3+\alpha \sigma R_4 z}{R_3-\sigma R_4 z}\,R_4zf'(R_4z)\right) \ne 0, \quad 0<|z|<1,
    $$
    where $R_3:=2-\sqrt{3}$ and $R_4$ $(\approx 0.19191)$ is the smallest positive root of the equation 
    $$
     \sin^{-1}\left(\frac{r}{R_3}\right)+4\sin^{-1}(r)-\frac{\pi}{2}=0.
    $$
\end{lem}
\begin{pf}
  According to Lemma~\ref{P3lem2}, it is sufficient to prove that for each complex number $\sigma\in\partial \D$ and $\alpha\in\partial \D$, there exists a constant $\beta\in\partial \D$ such that
  \begin{equation}\label{P3eq3.1}
      {\rm{Re}}\,\left((R_3-\beta z)\frac{R_3+\alpha \sigma z}{R_3-\sigma z}f'(z)\right)>0, \quad |z|<R_4,
  \end{equation}
  where $R_3$ and $R_4$ are defined in the statement of the lemma. A straightforward calculation yields
  \begin{equation}\label{P3eq3.2}
      \left|\arg{\left((R_3+\alpha \sigma z)f'(z)\right)}\right| \leq \sin^{-1}\left(\frac{r}{R_3}\right)+\left|\arg{f'(z)}\right|, \quad |z|=r<1.
  \end{equation}
  Since $f\in\mathcal{S}$, we have $\left|\arg{f'(z)}\right|\leq 4 \sin^{-1}(|z|)$ for $|z|\leq 1/\sqrt{2}$ (see \cite[Theorem~3.2.6]{graham}). Thus, from \eqref{P3eq3.2} we get
  \begin{equation*}
      \left|\arg{\left((R_3+\alpha \sigma z)f'(z)\right)}\right| \leq \sin^{-1}\left(\frac{r}{R_3}\right)+4\sin^{-1}(r), \quad |z|=r<\frac{1}{\sqrt{2}}.
  \end{equation*}
  The right hand side of the above inequality is less than $\pi/2$ if $|z|<R_4\approx 0.19191$. Thus,
  $$
   {\rm{Re}}\,\left((R_3+\alpha \sigma z)f'(z)\right)>0, \quad |z|<\min\left\{R_4,\frac{1}{\sqrt{2}}\right\}=R_4.
  $$
  Therefore, if we consider $\beta=\sigma$, then from the above inequality we get
    $$
      {\rm{Re}}\,\left((R_3-\beta z)\frac{R_3+\alpha \sigma z}{R_3-\sigma z}f'(z)\right)>0, \quad |z|<R_4.
    $$
    Hence, there exists a constant $\beta \in \partial \D$ such that the inequality \eqref{P3eq3.1} holds. This proves the lemma.
\end{pf}

We now present the main result of this section.
\begin{thm}\label{P3thm2}
    The radius of convexity of the class $\mathcal{S}\ast\mathcal{S}t(1/2)$ is at least $R_4\approx 0.19191$, where $R_4$ is defined as in \textnormal{Lemma~\ref{P3lem4}}.
\end{thm}
\begin{pf}
    It is easy to see that for $f\in\mathcal{S}$ and $\phi \in \mathcal{S}t(1/2)$,
    \begin{equation}\label{P3eq3.3}
       1+\frac{Rz(f\ast \phi)''(Rz)}{(f \ast \phi)'(Rz)}=\frac{\phi(z)\ast \big((FG)(Rz)\big)}{\phi(z)\ast G(Rz)}, \quad 0<|z|<1,
    \end{equation}
    where $0<R<1$, 
    $$
      F(z):=1+\frac{zf''(z)}{f'(z)}, \quad z \in \D,
    $$
    and
    $$
      G(z):=zf'(z), \quad z \in \D.
    $$
    Since $f\in\mathcal{S}$ and the radius of convexity of $\mathcal{S}$ is $2-\sqrt{3}$ (see \cite[Theorem~2.2.22]{graham}), we have
    \begin{equation}\label{P3eq3.4}
     {\rm{Re}}\,F(z)>0, \quad |z|<R_3:=2-\sqrt{3}.
    \end{equation}
    Let $w=R_3 z$, $z\in\D$. Since $|w|<R_3$, from \eqref{P3eq3.4} we get
    $$
     {\rm{Re}}\,F(w)>0,
    $$
    which implies
    $$
     {\rm{Re}}\,F(R_3z)>0, \quad z\in\D.
    $$
    Since $F(0)=1$, by the Herglotz formula we have
    $$
     F(R_3z)=\int\limits_{\partial \D} \frac{1+\sigma z}{1-\sigma z} d\mu (\sigma), \quad z\in\D,
    $$
    where $\mu$ is a probability measure on $\partial \D$. The above equation can be rewritten as
    $$
     F(z)=\int\limits_{\partial \D} \frac{R_3+\sigma z}{R_3-\sigma z} d\mu (\sigma), \quad |z|<R_3.
    $$
    By virtue of simple computations, from the above equation we get
    \begin{equation*}
        \frac{\phi(z)\ast \big((FG)(Rz)\big)}{\phi(z)\ast G(Rz)}=\frac{1}{\phi(z)\ast G(Rz)}\int\limits_{\partial \D}\phi(z) \ast \left(\frac{R_3+\sigma Rz}{R_3-\sigma Rz}\,G(Rz)\right)d\mu (\sigma), \quad z\in\D,
    \end{equation*}
    where $0<R\leq R_3$. Thus, from \eqref{P3eq3.3} and the above equation we get
    \begin{equation}\label{P3eq3.5}
        1+\frac{Rz(f\ast \phi)''(Rz)}{(f \ast \phi)'(Rz)}=\frac{1}{\phi(z)\ast G(Rz)}\int\limits_{\partial \D}\phi(z) \ast \left(\frac{R_3+\sigma Rz}{R_3-\sigma Rz}\,G(Rz)\right)d\mu (\sigma).
    \end{equation}
    By Lemmas~\ref{P3lem1} and \ref{P3lem4}, it follows that 
    $$
     {\rm{Re}}\, \left(\frac{\phi(z) \ast \left(\frac{R_3+\sigma R_4z}{R_3-\sigma R_4z}\,G(R_4z)\right)}{\phi(z)\ast G(R_4z)}\right)>0, \quad z\in\D,
    $$
    where $R_4$ is defined in Lemma~\ref{P3lem4}. Using this, from \eqref{P3eq3.5} we get
    $$
     {\rm{Re}}\,\left(1+\frac{R_4z(f\ast \phi)''(R_4z)}{(f \ast \phi)'(R_4z)}\right)>0, \quad z\in\D,
    $$
    i.e.
    $$
     {\rm{Re}}\,\left(1+\frac{z(f\ast \phi)''(z)}{(f \ast \phi)'(z)}\right)>0, \quad |z|<R_4\approx 0.19191.
    $$
    Therefore, $f\ast \phi$ is convex in $|z|<R_4\approx 0.19191$ for every $f\in\mathcal{S}$ and for every $\phi\in\mathcal{S}t(1/2)$. This completes the proof of the theorem.
\end{pf}

\begin{rem}
    If we set $A=2$, then the lower bound for the radius of convexity of $Co(A)\ast\mathcal{S}t(1/2)$ obtained in Theorem~\ref{P3thm1} reduces to the least positive root of the equation
    $$
      sin^{-1}\left(\frac{r}{2-\sqrt{3}}\right)+4sin^{-1}\left(r\right)-\frac{\pi}{2}=0,
    $$
    which coincides with the corresponding lower bound for the radius of convexity of $\mathcal{S}\ast\mathcal{S}t(1/2)$ obtained in Theorem~\ref{P3thm2}.

\end{rem}

\vspace{1cm}
\noindent{\bf Statements and  Declarations:}\\
 \noindent{\bf Competing interests:}\, The authors have no competing interests to declare that are relevant to the content of this article.

\end{document}